\documentclass[a4paper, 12pt]{article}
\usepackage{}

\usepackage{mathrsfs}
\usepackage{epsfig}
\usepackage{amsmath}
\usepackage{amssymb}
\usepackage{latexsym}
\usepackage{amsfonts}
\usepackage{amsthm}

\usepackage[numbers,sort&compress]{natbib}
\usepackage{graphicx}
\ifx\pdfoutput\undefined  \else
 \fi

\usepackage[centerlast]{caption2}

\usepackage{color}

\usepackage{txfonts}
\usepackage{amssymb}
\usepackage{mathrsfs}
\usepackage{amsfonts}

\newtheorem{theorem}{Theorem}[section]
\newtheorem{lemma}[theorem]{Lemma}
\newtheorem{cor}[theorem]{Corollary}
\newtheorem{prop}[theorem]{Proposition}
\newtheorem{conj}[theorem]{Conjecture}
\newtheorem{remark}[theorem]{Remark}

\usepackage{latexsym,amssymb}
\usepackage{amsmath}

\newcommand{\F}{{\mathbb F}}

\begin{document}

\title{{Existence of Erd\H{o}s-Burgess constant in commutative rings}}
\author{
Guoqing Wang\\
\small{School of Mathematics Science, Tiangong University, Tianjin, 300387, P. R. China}\\
\small{Email: gqwang1979@aliyun.com}
\\
}
\date{}
\maketitle

\begin{abstract}
Let $R$ be a commutative unitary ring. An idempotent in $R$ is an element $e\in R$ with $e^2=e$.  The Erd\H{o}s-Burgess constant associated with the ring $R$ is the smallest positive integer $\ell$ (if exists) such that for any given $\ell$ elements (not necessarily distinct) of $R$, say $a_1,\ldots,a_{\ell}\in R$, there must exist a nonempty subset $J\subset \{1,2,\ldots,\ell\}$ with $\prod\limits_{j\in J} a_j$ being an idempotent. In this paper, we prove that except for an infinite commutative ring with a very special form, the Erd\H{o}s-Burgess constant of the ring $R$ exists if and only if $R$ is finite.
\end{abstract}



\section {Introduction}

Let $\mathcal{S}$ be a nonempty commutative semigroup, endowed with a binary associative operation $*$. Let ${\rm E}(\mathcal{S})$ be the set of idempotents of $\mathcal{S}$, where $e\in \mathcal{S}$ is said to be an idempotent if $e*e=e$.
Idempotent is one of central notions in Semigroup Theory and Algebra,
also connects closely with other fields, see \cite{Cohen, Green} for the idempotent theorem in harmonic analysis, see \cite{Lint} for the application in coding theory.
One of our interest to combinatorial properties concerning idempotents in semigroups comes from a question of P. Erd\H{o}s to D.A. Burgess (see \cite{Burgess69} and \cite{Gillam72}), which can be restated as follows.

{\sl Let $\mathcal{S}$ be a finite nonempty semigroup of order $n$. A sequence of terms from $\mathcal{S}$ of length $n$ must contain one or more terms whose product, in some order, is idempotent?}

Burgess \cite{Burgess69} in 1969 gave an answer to this question in the case when $\mathcal{S}$ is commutative or contains only one idempotent. D.W.H. Gillam, T.E. Hall and N.H. Williams \cite{Gillam72} proved that a sequence $T$ over any finite semigroup $\mathcal{S}$ of length at least $|\mathcal{S}\setminus{\rm E}(\mathcal{S})|+1$ must contain one or more terms whose product, in the order induced from the sequence $T$, is an idempotent, and therefore, completely answered Erd\H{o}s' question. The Gillam-Hall-Williams Theorem was extended to infinite semigroups by the author \cite{wangidempotent}.
It was also remarked  that the bound $|\mathcal{S}\setminus {\rm E}(\mathcal{S})|+1$,  although is
optimal for general semigroups $\mathcal{S}$, can be improved, at least in principle, for specific classed of semigroups.
Naturally, one combinatorial invariant was aroused by Erd\H{o}s' question with respect to idempotents of semigroups. Since we deal with the multiplicative semigroup of a commutative ring in this paper, we introduce only the definition of this invariant for commutative semigroups here.

\noindent \textbf{Definition.} (\cite{wangidempotent}, Definition 4.1) \ {\sl For a commutative semigroup $\mathcal{S}$, we define the {\bf Erd\H{o}s-Burgess constant} of $\mathcal{S}$, denoted by $\textsc{I}(\mathcal{S})$, to be the smallest positive integer $\ell$ (if exists) such that every sequence $T$ of terms from $\mathcal{S}$ and of length $\ell$ must contain one or more terms whose product is an idempotent. If no such integer $\ell$ exists, we let $\textsc{I}(\mathcal{S})=\infty$.}

Note that if the commutative semigroup $\mathcal{S}$ is finite, Gillam-Hall-Williams Theorem definitely tells us that the Erd\H{o}s-Burgess constant of  $\mathcal{S}$ exists (i.e., $\textsc{I}(\mathcal{S})$ is finite) and bounded above by $|\mathcal{S}\setminus{\rm E}(\mathcal{S})|+1$. In particular,
when the semigroup $\mathcal{S}$ happens to be a finite abelian group, the Erd\H{o}s-Burgess constant reduces to a classical combinatorial invariant, the Davenport constant.
The Davenport constant of a finite abelian group $G$, denoted ${\rm D}(G)$, is defined as the smallest positive integer $\ell$  such that every sequence of terms from $G$ of length at least $\ell$ contains one or more terms with the product being the identity element of $G$. This invariant was popularized by H. Davenport in the 1960's, notably for its link with algebraic number theory (as reported in \cite{Olson1}). Davenport constant has been investigated extensively in the past over 50 years, and  found applications in other areas, including Factorization Theory of Algebra \cite{GRuzsa, GH}, Classical Number Theory \cite{AGP}, Graph Theory \cite{AlonJctb}, and Coding Theory \cite{MaOrSaSc}.
What is more important, a lot of researches were motivated by the Davenport constant together with the celebrated EGZ Theorem obtained by P. Erd\H{o}s, A. Ginzburg and A. Ziv \cite{EGZ} in 1961 on additive properties of sequences in groups, which have been developed into a branch, called zero-sum theory (see e.g. \cite{GaoGeroldingersurvey}, and \cite{Grynkiewiczmono} for a survey), in Combinatorial Number Theory.
Recently  some zero-sum type problems were also investigated in the setting of commutative semigroups
(see e.g. \cite{wangDavenportII, wangAddtiveirreducible, wang-zhang-qu, wang-zhang-wang-qu}).

To investigate the Erd\H{o}s-Burgess constant associated with commutative rings,  one fundamental question remains:

{\bf When does the Erd\H{o}s-Burgess constant exist for a commutative ring? }

In this paper, we shall answer this question by proving the following theorem.

\bigskip

\bigskip

\begin{theorem} \label{theorem equivalent I(S) finite in rings} \ Let $R$ be a commutative unitary ring, and let $\mathcal{S}_R$ be the the multiplicative semigroup of $R$. If $\textsc{I}(\mathcal{S}_R)$ is finite, then one of the following two conditions holds:

\noindent (i) The ring $R$ is finite;

\noindent (ii) The Jacobson radical $J(R)$ is finite and $R\diagup J(R)\cong B \times \prod\limits_{i=1}^t \mathbb{F}_{q_i}$, where  $B$ is an infinite Boolean unitary ring, and $\mathbb{F}_{q_1},\ldots, \mathbb{F}_{q_t}$ are finite fields with $0\leq t\leq \textsc{I}(\mathcal{S}_R)-1$ and prime powers $q_1,\ldots, q_t>2$.
\end{theorem}

\noindent \textbf{Remark.} \ Recall that by Gillam-Hall-Williams Theorem, if the ring $R$ is finite then $\textsc{I}(\mathcal{S}_R)$ exists. Hence, Theorem \ref{theorem equivalent I(S) finite in rings} asserts that the Erd\H{o}s-Burgess constant exists only for {\bf finite} commutative rings except for an infinite commutative rings with a very special form given as (ii). That is, to study this invariant in the realm of commutative rings, we may consider it only for finite commutative rings.

\section{The Proof of Theorem \ref{theorem equivalent I(S) finite in rings}}

For integers $a,b\in \mathbb{Z}$, we set $[a,b]=\{x\in \mathbb{Z}: a\leq x\leq b\}$.
Let $(R,+,*)$ be a commutative unitary ring, and let $T$ be a sequence of terms from $R$. By $|T|$ we denote the length of the sequence $T$. We call $T$ an {\sl idempotent-product free} sequence provided that no idempotent of $R$ can be represented as a product of one or more terms from $T$. By the  definition, we have immediately that the Erd\H{o}s-Burgess constant $\textsc{I}(\mathcal{S}_R)$ exists if and only if ${\rm sup}\ \{|T|: T \mbox{ is taken over all idempotent-product free sequences over } R\}$ is finite.

\begin{lemma} \label{lemma folklore} Let $G$ be an abelian group. Then ${\rm I}(G)$ is finite if and only if $G$ is finite.
\end{lemma}

\begin{proof}  Since the identity element is the unique idempotent in $G$, the sufficiency of the lemma is well-known in zero-sum theory and follows from a simple application of the pigeonhole principle. Now we show the necessity. Suppose $|G|$ is infinite. Let $T=(a_1,\ldots, a_n)$ be an arbitrary idempotent-product free sequence over $G$. By the infinity of $|G|$, we can find a nonidentity element $g\in G$ such that $g^{-1}$, the inverse of $g$, can not be represented as the product of one or more terms from $T$.  We see that the sequence $(a_1,\ldots, a_n, g)$ obtained by adjoining the element $g$ to $T$ is idempotent-product free. By the arbitrariness of $T$, we conclude that ${\rm I}(G)$ is infinite, completing the proof.
\end{proof}

\begin{lemma}\label{lemma I(sub)< I(S)} Let $S$ be a commutative semigroup and $S'$ a subsemigroup of $S$. If ${\rm I}(\mathcal{S})$ is finite, then ${\rm I}(\mathcal{S}')$ is finite and ${\rm I}(\mathcal{S}')\leq {\rm I}(\mathcal{S})$.
\end{lemma}

\begin{proof} The conclusion follows immediately from the fact that any idempotent-product free sequence of terms from $S'$ is also an idempotent-product free sequence of terms from $S$.
\end{proof}

\begin{lemma}\label{lemma I(im)< I(S)} Let $S$ and $S'$ be commutative semigroups. If there is an epimorphism $\varphi$ of $S$ onto $S'$, then
$\textsc{I}(S')\leq \textsc{I}(S)$.
\end{lemma}

\begin{proof} Let $T'=(b_1,\ldots,b_{\ell})$ be an arbitrary idempotent-product free sequence of terms from $S'$. We can take a sequence $T=(a_1,\ldots,a_{\ell})$ of terms from $S$ such that $\varphi(a_i)=b_i$ for each $i\in [1,\ell]$. Since the epimorphism $\varphi$ always maps an idempotent of  $S$ to an idempotent of $S'$, we have that the sequence $T$ is also idempotent-product free in $S$. By the arbitrariness of $T'$, we derive that $\textsc{I}(S')\leq \textsc{I}(S)$.
\end{proof}

\begin{lemma} (see \cite{Neal H McCoy}, Theorem 3.9) \label{lemma respentation as subdirect sum} \ A ring $R$ has a representation as a subdirect sum of rings $S_i$, $i\in \mathcal{A}$, if and only if for each $i\in \mathcal{A}$ there exists in $R$ a two-sided ideal $K_i$ such that $R\diagup K_i\cong S_i$ and, moreover, $\bigcap\limits_{i\in \mathcal{A}} K_i=(0_R)$.
\end{lemma}

\begin{lemma} (see \cite{Neal H McCoy}, Theorem 3.16) \ \label{lemma condition to be boolean ring} \ A ring is isomorphic to a subdirect sum of fields $\mathbb{F}_2$ is and only it is a Boolean ring.
\end{lemma}

\begin{lemma}  \label{lemma R/N is boolean ring} \ Let $R$ be a commutative unitary ring. Let $\{M_i: i\in \mathcal{A}\}$ be a family (nonempty) of maximal ideals of $R$ with index two. Then $R\diagup \bigcap\limits_{i\in \mathcal{A}} M_i$ is a Boolean unitary ring.
\end{lemma}

\begin{proof} \ Let
\begin{equation}\label{equation N=cap Mi}
N=\bigcap\limits_{i\in \mathcal{A}} M_i.
\end{equation}
 We see that $M_i\diagup N$ are distinct maximal ideals of $R\diagup N$ with index $[R\diagup N: M_i\diagup N]=[R: M_i]=2$, and so
\begin{equation}\label{equation R/N Mi/N=F2}
\frac{R\diagup N}{M_i\diagup N}\cong \mathbb{F}_2,
\end{equation}
 where $i\in \mathcal{A}$. By \eqref{equation N=cap Mi}, we derive that
 \begin{equation}\label{equation cap=0}
 \bigcap_{i\in \mathcal{A}}(M_i\diagup N)=(0_{R\diagup N}).
 \end{equation}
 By \eqref{equation R/N Mi/N=F2}, \eqref{equation cap=0}, Lemma \ref{lemma respentation as subdirect sum} and Lemma \ref{lemma condition to be boolean ring}, we derive that $R\diagup N$ is a Boolean unitary ring.
 \end{proof}

\noindent {\bf Proof of Theorem \ref{theorem equivalent I(S) finite in rings}.} \ Suppose
\begin{equation}\label{equation I(S_R)=n}
\textsc{I}(\mathcal{S}_R)=n
\end{equation}
 is finite and
\begin{equation}\label{equation SrsetminuE(Sr)=infinity}
|R|=\infty.
\end{equation}
It suffices to prove (ii) holds. Since the group ${\rm U}(R)$ is a subsemigroup of
$\mathcal{S}_R$ where ${\rm U}(R)$ denotes the group of units of the ring $R$, it follows from \eqref{equation I(S_R)=n} and Lemma \ref{lemma I(sub)< I(S)} that $\textsc{I}({\rm U}(R))\leq n.$ By Lemma \ref{lemma folklore}, we derive that
$|{\rm U}(R)|<\infty.$ Since $1_R+J(R)\subset {\rm U}(R)$,  it follows that
\begin{equation}\label{equation J(R) is finite}
|J(R)|<\infty.
\end{equation}

\noindent \textbf{Claim A.} \ The index of each maximal ideal in $R$ is finite.

\noindent {\sl Proof of Claim A.} \ Assume to the contrary that there exists some maximal ideal $M$ such that the index of $M$ in $R$ is infinite, i.e., $R\diagup M$ is an infinite field. Since the group ${\rm U}(R\diagup M)$ is a subsemigroup of
$\mathcal{S}_{R\diagup M}$ and there is a canonic epimorphism of the semigroup $\mathcal{S}_R$ onto $\mathcal{S}_{R\diagup M}$ with rings' multiplication of $R$ and $R\diagup M$, it follows from \eqref{equation I(S_R)=n},  Lemma \ref{lemma I(sub)< I(S)} and Lemma \ref{lemma I(im)< I(S)} that
$\textsc{I}({\rm U}(R\diagup M))\leq \textsc{I}(\mathcal{S}_{R\diagup M}))\leq \textsc{I}(\mathcal{S}_R)=n$. Combined with Lemma \ref{lemma folklore}, we have that $|{\rm U}(R\diagup M)|<\infty$ and so $|R\diagup M|=|{\rm U}(R\diagup M)|+1<\infty$, a contradiction. This proves Claim A. \qed

\noindent \textbf{Claim B.} \ The ring $R$ has at most $n-1$ maximal ideals with index greater than two.

\noindent{\sl Proof of Claim B.} \ Assume to the contrary that there exists at least $n$ distinct maximal ideals, say $M_1,\ldots,M_n$, of $R$ with index greater than two. Combined with Claim A, we see that $R\diagup M_i$ is a finite field of order $|R\diagup M_i|>2$, which implies that $|{\rm U}(R\diagup M_i)|\geq 2$ and so the group ${\rm U}(R\diagup M_i)$ contains at least one non-idempotent element, for each $i\in [1,n]$. Therefore,  there are $b_1,b_2,\ldots,b_n$ (not necessarily distint) of $R$ such that $b_i^2\not\equiv b_i\pmod {M_i}$ for each $i\in [1,n]$. By the Chinese Remainder Theorem, we can find $a_1,\ldots,a_n$ of $R$ such that $a_i\equiv b_i \pmod {M_i}$ and $a_i\equiv 1_R\pmod {M_j}$ for $j\in [1,n]\setminus \{i\}$, where $i\in [1,n]$. Let $L$ be the sequence consisting of exactly all these terms $a_1,\ldots,a_n$. We check that the sequence $L$ is idempotent-product free, which implies that $\textsc{I}(\mathcal{S}_R)\geq |L|+1=n+1$, a contradiction with \eqref{equation I(S_R)=n}. This proves Claim B. \qed

\noindent \textbf{Claim C.} \ The ring $R$ has infinitely many maximal ideals with index two.

\noindent {\sl Proof of Claim C.} \ Assume to the contrary that there exists only finitely many maximal ideals with index two.
Combined with Claim A and Claim B, we derive that $R$ has only finitely many maximal ideals. Since $J(R)=\bigcap\limits_{ M \mbox{ ranges over all maximal ideals}} M$, it follows from the Chinese Remainder Theorem that
$R\diagup J(R)\cong \prod\limits_{M \mbox{ ranges over all maximal ideals}} R\diagup M$. Combined with Claim A,  we derive that $|R\diagup J(R)|$ is finite. By \eqref{equation J(R) is finite}, we derive that $R$ is finite, which is a contradiction with \eqref{equation SrsetminuE(Sr)=infinity}.  This proves Claim C. \qed

Let
$N=\bigcap_{i\in \mathcal{A}} M_i$
where $\{M_i: i\in \mathcal{A}\}$ is the set of all maximal ideals of $R$ of index two.
Take a representation
\begin{equation}\label{equation J(R)=N cap K}
J(R)=N\cap K_1\cap \cdots \cap K_t
\end{equation}
 such that $t\geq 0$ is {\bf minimal}, where $K_1,\ldots,K_t$ are distinct maximal ideals of $R$ of index greater than two. By the minimality of $t$, we conclude that $N\nsubseteq K_i$ for each $i\in [1,t]$ and so $N, K_1, \ldots, K_t$ are pairwise coprime. By \eqref{equation J(R)=N cap K} and the Chinese Remainder Theorem, we derive that $R\diagup J(R)\cong (R\diagup N)\times (\prod\limits_{i=1}^t R\diagup K_i).$ By Claim A, we derive that there exists primes powers
$q_1,\ldots, q_t>2$
  such that $R\diagup K_i\cong \mathbb{F}_{q_i}$ for each $i\in [1,t]$. i.e.,
 \begin{equation}\label{equation R/J(R)=R/N time some Fqi}
 R\diagup J(R)\cong (R\diagup N)\times (\prod\limits_{i=1}^t \mathbb{F}_{q_i}).
\end{equation}
By Lemma \ref{lemma R/N is boolean ring}, we have $R\diagup N$ is a Boolean unitary ring. By \eqref{equation SrsetminuE(Sr)=infinity} and \eqref{equation J(R) is finite}, we see $|R\diagup J(R)|$ is infinite.
 Combined with \eqref{equation R/J(R)=R/N time some Fqi}, we derive that $|R\diagup N|$ is infinite. Combined with \eqref{equation I(S_R)=n} and Claim B, $t\leq \textsc{I}(\mathcal{S}_R)-1$ and (ii) holds readily. This completes the proof of the theorem.  \qed

 As a consequence of Theorem \ref{theorem equivalent I(S) finite in rings}, we have the following.

\begin{cor} \ If $R$ is a commutative Noetherian unitary ring. Then $\textsc{I}(\mathcal{S}_R)$ is finite if and only if $R$ is finite.
\end{cor}

\begin{proof} \ Since any infinite Boolean ring is not Noetherian (see \cite{Clark}, Proposition 9.6), we could derive that the ring $R$ meeting Condition (ii) of Theorem \ref{theorem equivalent I(S) finite in rings} is not Noetherian. Then the conclusion follows immediately.
\end{proof}

\begin{prop}\label{prop moreinformation} \ Let $R$ be an infinite commutative unitary ring. If $\textsc{I}(\mathcal{S}_R)$ is finite, then $R$ has infinitely many maximal ideals of index two and has at most $\textsc{I}(\mathcal{S}_R)-1$ maximal ideals with index greater than two and has no maximal ideals of infinite index.
\end{prop}

We remark that Proposition \ref{prop moreinformation} can be derived from the arguments of Theorem \ref{theorem equivalent I(S) finite in rings}.  However, to show that Theorem \ref{theorem equivalent I(S) finite in rings} itself implies Proposition \ref{prop moreinformation}, we give a short proof here.

\begin{proof} \  By Theorem \ref{theorem equivalent I(S) finite in rings}, $R\diagup J(R)\cong B \times \prod\limits_{i=1}^t \mathbb{F}_{q_i}$, where $B$ is an infinite Boolean unitary ring,  $\mathbb{F}_{q_1},\ldots, \mathbb{F}_{q_t}$  ($t\geq 0$, $q_1,\ldots,q_t>2$) are finite fields. Note that infinite Boolean unitary ring $B$ has infinitely many maximal ideals, and each of the maximal ideals has index two (see \cite{Clark}, Proposition 9.4 and Proposition 9.6). Since $B$ has an identity, any ideal $K\triangleleft R\diagup J(R)$ must be of the form $K=K_0\times K_1\times \cdots\times K_t$ where $K_0\triangleleft B, K_1\triangleleft \mathbb{F}_{q_1},\ldots, K_t\triangleleft \mathbb{F}_{q_t}$. We derive that $R\diagup J(R)$ has infinitely many maximal ideals of index two and has exactly $t\leq \textsc{I}(\mathcal{S}_R)-1$ maximal ideals with index greater than two, in precise, with indices $q_1,\ldots,q_t>2$ respectively, and has no maximal ideals of infinite index, thus, so does the ring $R$, since $J(R)$ is the intersection of all maximal ideals of $R$.
\end{proof}

From Proposition  \ref{prop moreinformation}, we have the following immediately.

\begin{cor} \ If $R$ is a commutative semi-local unitary ring, i.e., $R$ has only finitely many maximal ideals.  Then $\textsc{I}(\mathcal{S}_R)$ is finite if and only if $R$ is finite.
\end{cor}

We conjecture that the conditions of Theorem \ref{theorem equivalent I(S) finite in rings} should be also sufficient. Hence, we close this paper with the following conjecture.

\begin{conj} \ \label{conjecture rings} Let $R$ be a commutative unitary ring with  $R\diagup J(R)\cong B \times \prod\limits_{i=1}^t \mathbb{F}_{q_i}$, where $B$ is an infinite Boolean unitary ring,   $\mathbb{F}_{q_1},\ldots, \mathbb{F}_{q_t}$  ($t\geq 0$) are finite fields, and the Jacobson radical $J(R)$ is finite. Then
$\textsc{I}(\mathcal{S}_R)$ is finite.
\end{conj}

\begin{remark} \ Note that if the ring $R$ has a zero Jacobson Radical $J(R)=(0_R)$ ($R$ is called Jacobson-semisimple), then $R\cong B \times \prod\limits_{i=1}^t \mathbb{F}_{q_i}$, it is not hard to show that $\textsc{I}(\mathcal{S}_R)=\textsc{I}(\mathcal{S}_{\prod\limits_{i=1}^t \mathbb{F}_{q_i}})$ which is finite and Conjecture \ref{conjecture rings} holds true.
\end{remark}

\noindent {\bf Acknowledgements}

\noindent
This work is supported by NSFC (grant no. 11971347, 11501561).

\end{document}